\documentclass[a4paper,10pt,reqno]{amsart}

\usepackage[T1]{fontenc}
\usepackage[utf8]{inputenc}
\usepackage{lmodern}
\usepackage[english]{babel}
\usepackage[
    a4paper,
	left=2.5cm,       
	right=2.5cm,      
	top=3.5cm,        
	bottom=3.5cm,     
	heightrounded,    
	bindingoffset=0mm 
]{geometry}

\usepackage{esint}
  
\usepackage{amsmath, xfrac}
\usepackage{amssymb}
\usepackage{mathtools}
\usepackage{mathrsfs}
\usepackage[dvipsnames]{xcolor}
\usepackage{graphicx}
\usepackage{enumitem}

\usepackage{hyperref} 
\usepackage{esint}
\usepackage{verbatim}

 \usepackage[parfill]{parskip} % Activate to begin paragraphs with an empty line rather than an indent

\usepackage[noabbrev,capitalize]{cleveref}
\usepackage{bookmark}
\usepackage{autonum}

\usepackage[initials]{amsrefs}

\usepackage{bbm}

\usepackage[]{hyperref}
\hypersetup{
    colorlinks=true,       % false: boxed links; true: colored links
    linkcolor=red,          % color of internal links
    citecolor=blue,        % color of links to bibliography
    filecolor=red,      % color of file links
    urlcolor=cyan           % color of external links
}

\numberwithin{equation}{section}

\renewcommand{\phi}{\varphi}

\theoremstyle{definition}
\newtheorem{Def}{Definition}[section]

\theoremstyle{plain}
\newtheorem{prop}[Def]{Proposition}
\newtheorem{cor}[Def]{Corollary}
\newtheorem{thm}[Def]{Theorem}
\newtheorem{lem}[Def]{Lemma}

\theoremstyle{remark}

\newcommand{\norm}[1]{\left\|#1\right\|}
\newcommand{\abs}[1]{\left|#1\right|}

\DeclareMathOperator{\divergence}{div}
\DeclareMathOperator{\curl}{curl}

\newcommand{\e}{\varepsilon}
\newcommand{\R}{\mathbb{R}}
\newcommand{\N}{\mathbb{N}}

\newcommand{\D}{\mathcal{D}}

\newcommand{\Tr}{{\rm Tr}}

\newcommand{\T}{\mathbb{T}}

\newcommand{\Spt}{\text{Spt}}

\renewcommand{\MR}[1]{} % remove MR numbers

\begin{document}

\title[The helicity distribution] {The helicity distribution for the 3D incompressible Euler equations}

\author[M.~Inversi]{Marco Inversi}
\address[M.~Inversi]{Departement Mathematik und Informatik, Universität Basel, Spiegelgasse~1, 4051 Basel, CH} 
\email{marco.inversi@unibas.ch}

\author[M.~Sorella]{Massimo Sorella}
\address[M.~Sorella]{Department of Mathematics, Imperial College London, SW7 2AZ, London, UK}
\email{m.sorella@imperial.ac.uk}



\begin{abstract}
This paper is concerned with the helicity associated to solutions of the 3D incompressible Euler equations. We show that under mild conditions on the regularity of the velocity field of an incompressible ideal fluid it is possible to define a defect distribution describing the local helicity balance. Under suitable regularity assumptions, we also provide the global helicity balance on bounded domains in terms of the boundary contributions of the vorticity, velocity and pressure.  
\end{abstract}

%\date{\today}

\maketitle

\section{Introduction} 
%{\textcolor{red}{ ADD citations: Even though you are considering an Onsager-type conjecture in a bounded domain, your have been very selective with your citations ignoring even  the very first work on the subject by Bardos and Titi (2018)  in ARMA on the Onsager conjecture with boundaries (also the CMP (2019) paper Bardos-Titi-Wiedemann). Moreover, and as I mentioned above, even though you need to use pressure regularity estimates in a bounded domain, you are not citing the very first paper on this subject Bardos-Titi (2022), as well as  Bardos-Boutros-Titi (arXiv); while here again you cite the follow up work of De Rosa et al. ignoring the original work.} }

In this paper we consider the Euler equations in $\Omega \times (0,T)$ for the velocity field $u$ of an ideal incompressible fluid in a spatial domain $\Omega$, that is 
\begin{align} \label{Euler} \tag{E}
    \begin{cases}
        \partial_t u + \divergence(u \otimes u) + \nabla p = 0 & (x,t ) \in \Omega \times (0,T),  
        \\ \divergence(u) = 0 & (x,t) \in \Omega \times (0,T). 
    \end{cases}
\end{align}  
Here $\Omega$ is either $\T^3, \R^3$ or any bounded domain $\Omega \subset \R^3$ with Lipschitz boundary. If $\partial \Omega \neq \emptyset$, the system \eqref{Euler} is coupled with the impermeability boundary condition $u(t) \cdot n = 0$ at $\partial \Omega$ for any $t \in (0,T)$, the latter to be understood in a suitable way according to the regularity of the velocity field. Throughout this note, $n$ is the outer unit normal to $\partial \Omega$. Weak solutions to \eqref{Euler} are defined in the usual distributional sense. 

\begin{Def} \label{d: weak solution}
Let $\Omega \subset \R^3$ be any open set and $u \in L^2_{\rm loc}(\Omega \times (0,T)), p \in L^1_{\rm loc}(\Omega \times (0,T))$. We say that $(u,p)$ is a weak solution to \eqref{Euler} if 
\begin{align}
    \int_0^T \int_{\Omega} \left[ \partial_t \phi \cdot u + u \otimes u \colon \nabla \phi + p \divergence(\phi) \right] \, dx \, dt & = 0 \qquad \forall \phi \in C^\infty_c(\Omega \times (0,T); \R^3), 
    \\ \int_\Omega u(x,t) \cdot \nabla \psi(x) \, dx & = 0 \qquad \forall \psi \in C^\infty_c(\Omega), \, \text{ for a.e. } t \in (0,T). 
\end{align}
\end{Def}

The boundary condition for weak solutions according to \cref{d: weak solution} can be interpreted in the sense of normal distributional traces. Indeed, if $\Omega$ is a bounded open set and $u(\cdot,t) \in L^1(\Omega)$ for a.e. $t$, the system \eqref{Euler} is complemented with the \emph{impermeability boundary condition} that is 
\begin{equation} \label{eq: impermeability condition}
    \int_{\Omega} u(x,t) \cdot \nabla \psi(x) \, dx \qquad \forall \psi \in C^\infty_c(\R^3), \, \text{ for a.e. } t \in (0,T). 
\end{equation}
By the divergence theorem, it is clear  that \eqref{eq: impermeability condition} is equivalent to $u(\cdot, t) \cdot n = 0$ at $\partial \Omega$ whenever $u(\cdot, t)$ is smooth up to the boundary and divergence-free. Moreover, given a regular solution $u$ to \eqref{Euler}, it is well-known that the vorticity $\omega : = \curl(u)$ satisfies 
\begin{equation} \label{euler vorticity} \tag{V}
    \partial_t \omega + (u \cdot \nabla) \omega - (\omega \cdot \nabla) u = 0 \qquad (x,t) \in \Omega \times (0,T) 
\end{equation} 
and, letting the \emph{helicity} be the scalar product between the velocity and the vorticity, a direct computation shows that 
\begin{equation}
    \partial_t ( u \cdot \omega) + \divergence \left( (u \cdot \omega) u + \left( p - \frac{\abs{u}^2}{2} \right) \omega  \right) = 0 \qquad (x,t) \in \Omega \times (0,T). \label{local helicity smooth}
\end{equation}
Motivated by Duchon--Robert \cite{DR00}, the main goal of this paper is to show that \eqref{local helicity smooth} can be defined for weak solutions of \eqref{Euler} up to a defect distribution $D[u]$, under mild regularity assumptions on the velocity field. Moreover, the distribution $D[u]$ can be characterized as a weak limit of regular functions, as for the Duchon--Robert measure describing the possible failure of the local energy balance. We state the first  result of this paper (see \cref{ss: fractional sobolev space} for a precise definition of the function spaces involved). 

\begin{thm}[Duchon--Robert type helicity distribution] \label{th:main-helicity}
Let $\Omega \subset \R^3$ be any open set and let $(u,p)$ be a weak solution to \eqref{Euler} in $\Omega \times (0,T)$ according to \cref{d: weak solution} with $u\in L^2_{\rm loc} ( (0,T) ; H^{\sfrac{1}{2}}_{\rm loc} (\Omega)) \cap L^\infty_{\rm loc}( \Omega \times (0,T))$. Then, there exists $D[u] \in \D'(\Omega \times (0,T))$ such that 
\begin{equation}
    \partial_t ( u \cdot \omega) + \divergence \left( (u \cdot \omega) u + \left( p - \frac{\abs{u}^2}{2} \right) \omega  \right) = - D[u] \qquad \text{ in } \mathcal{D}'(\Omega \times (0,T)). \label{local helicity rough}
\end{equation}
Furthermore, letting $u_\e, (u \otimes u)_\e, \omega_\e$ be the spatial mollifications of $u, u \otimes u, \omega$ respectively, it holds 
\begin{equation}
    D[u] = 2 \lim_{\e \to 0}  \nabla \omega_\e : R_\e,
\end{equation} 
where $ R_\e : = (u_\e \otimes u_\e) - (u \otimes u)_\e $ and the limit is intended in $\mathcal{D}'(\Omega \times (0,T))$. 
\end{thm}

If $\Omega = \T^3$ the statement and the proof of \cref{th:main-helicity} are modified accordingly. In this case, \eqref{local helicity rough} can be tested with any periodic smooth function on $\T^3$, since $\T^3$ has no boundary. We highlight that a similar result in the periodic setting has been obtained simultaneously in \cite{BT24}*{Theorem 2.7} with the assumption  $u \in L^3_t W^{\sfrac{1}{3}+, 3}_x$. We briefly comment on our assumption $u \in L^2_t H^{\sfrac{1}{2}}_x \cap L^\infty_{x,t}$.  If $u \in L^2_t H^{\sfrac{1}{2}}_x$, then $\omega \in L^2_t H^{-\sfrac{1}{2}}_x$ and the coupling $u \cdot \omega$ can be defined as a distribution (see \cref{l: time dependent duality} and \cref{coupling in H^1/2 - H^-1/2}). If we assume in addition that $u \in L^\infty_{x,t}$, then it is straightforward to check that quadratic functions of $u$ belong to $ L^2_t H^{\sfrac{1}{2}}_x$ (see \cref{H12 is algebra}). The pressure term can be treated by Calderón--Zygmund estimates. In other words, the assumption $u \in L^2_t H^{\sfrac{1}{2}}_x \cap  L^\infty_{x,t}$ is natural from the equations to interpret \eqref{local helicity rough} in the sense of distributions. Moreover, by interpolation it is readily checked that $H^{\sfrac{1}{2}}\cap L^\infty$ embeds continuously into $B^{\sfrac{1}{3}}_{3, \infty}$, which is the critical space for the celebrated Onsager conjecture related to the conservation of the kinetic energy for weak solutions to \eqref{Euler}. The rigidity part of the conjecture has been established by Constantin--E--Titi and Eyink \cites{CET94, Ey94} by commutator type estimates similar to \cite{DPL89}, see also \cites{CCFS08, DR00, DRINV23} for related results. The flexible part has been proved in \cite{Is18} by Isett, building on the ground breaking ideas of convex integration introduced in the context of incompressible fluids by De Lellis and Székelyhidi \cites{DLL09,DLL13}.
For other convex integration results in incompressible fluid dynamics see also \cites{GKN23, NV23,BV19,CDRS22} and the references therein.

To the best of our knowledge, the investigation of similar questions for the helicity is still at a preliminary stage. It is known that $B^{\sfrac{2}{3}}_{3, \infty}$ is the critical space for the conservation of the total helicity \cites{CCFS08, DR19}. We are not aware of construction of solutions to \eqref{Euler} for which the total/local helicity is not conserved. To this end, it seems important to establish a local helicity balance (up to a defect distribution) as in \eqref{local helicity rough} for rough solutions of \eqref{Euler}. By this method, we prove the validity of the exact helicity balance (i.e. $D[u] \equiv 0$ in \eqref{local helicity rough}) under a $\sfrac{2}{3}$ fractional spatial regularity on the velocity field (see \cref{ss: commutator estimates and besov} for a precise definition of Besov spaces). 

\begin{cor} \label{cor: conservation}
Under the assumptions of \cref{th:main-helicity}, suppose that for any open set $I \subset \joinrel \subset (0,T), O \subset \joinrel \subset \Omega$ it holds that $u \in L^3(I; B^{\sfrac{2}{3}}_{3,c_0} (O))$ according to \cref{ss: commutator estimates and besov}. Then $D[u] \equiv 0$. 
\end{cor}

This result has essentially been proved in \cite{CCFS08}*{Theorem 4.2} on the periodic box $\T^3$ by the analysis of the energy flux with a Paley--Littlewood decomposition. Our approach is a mollification type argument inspired by \cite{DR19}, which is based on \cites{DPL89,DR00,CET94}. However, in \cite{DR19} the author proves conservation of helicity assuming  that $\curl (u)$ is integrable, together with further suitable assumptions. We remark that  both the space $B^{\sfrac{2}{3}}_{3, c(\N)}(\R^3)$ considered in \cite{CCFS08} and the space of $B^{\sfrac{2}{3}}_{3,c_0}(\R^3)$ defined in \cref{ss: commutator estimates and besov} can be characterized as the closure of smooth functions with respect to the $B^{\sfrac{2}{3}}_{3,\infty}(\R^3)$ norm, thus they coincide. %Indeed, on the periodic box $\T^d$, it turns out that Besov spaces can be equivalently defined with Paley--Littlewood decomposition or in terms of Gagliardo seminorm \cite{GL23}. 
Since we are interested in the \emph{local} helicity balance, it seems more natural to describe fractional differentiability in terms of the Gagliardo seminorm, but the approaches are completely equivalent.  

On a bounded domain $\Omega \subset \R^3$ the total helicity might not be conserved if the vorticity is not tangent to the boundary. Indeed, due to lack of a natural boundary condition on the vorticity, the boundary can be used in simulations to create nontrivial vorticity which forces the fluid to develop a turbulent behaviour inside the domain \cite{Fr95}*{Section 8.9}. For technical reasons, on a bounded open set we need additional assumptions on the vorticity to define the total helicity. 

\begin{Def}[Total helicity] \label{total helicity}
Let $\Omega$ be a bounded open set and let $u \in L^2(\Omega)$ with $\omega = \curl(u) \in L^2(\Omega)$. We define the total helicity by 
\begin{equation}
    H: = \int_{\Omega} \omega \cdot u \, dx. 
\end{equation}
\end{Def} 
It is possible to generalize this definition only requiring that $ u \in \mathbb{X}$, where $\mathbb{X}$ is a Banach space, and $\curl (u) \in (\mathbb{X})'$. For instance, this is the case on $\T^3$ whenever $u \in L^2 ((0,T); H^{\sfrac{1}{2}} (\T^3))$.

Under suitable regularity assumptions, we compute the variation of the total helicity in terms of the normal component of the vorticity and the full trace of the velocity and the pressure at the boundary. Our approach is based on the analysis of the \emph{normal Lebesgue boundary} trace given in \cites{DRINV23, CDRIN24}. 

\begin{thm}  \label{thm:conservation-bdd-domain}
Let $\Omega \subset \R^3$ be a simply connected bounded open set with smooth boundary and let $(u,p) \in L^\infty( \Omega \times (0,T))$ be a weak solution to \eqref{Euler} according to \cref{d: weak solution} with the impermeability boundary condition \eqref{eq: impermeability condition}. Assume that $\omega \in L^\infty(\Omega \times (0,T))$ has normal Lebesgue trace $\omega_n(t)$ for a.e. $t \in (0,T)$ according to \cref{def:normal}. Then, $u(\cdot, t), p(\cdot, t) \in C^0(\overline{\Omega})$ for a.e. $t \in (0,T)$ and the total helicity (see \cref{total helicity}) satisfies
    \begin{align} \label{eq:global-helicity-conservation}
        \int_0^T  H(t) \alpha'(t) \, dt  & = \int_0^T \alpha(t) \left[  \int_{\partial \Omega} \left( \frac{\abs{u}^2}{2} - p \right) \omega_n \, d \mathcal{H}^2(x) \right] \, dt \qquad \forall \alpha \in C^\infty_c((0,T)).
    \end{align}   
\end{thm}

We point out that the assumptions on $\partial \Omega,u,p$ and on the vorticity away from the boundary in \cref{thm:conservation-bdd-domain} could be weakened, similarly to those of Theorem \ref{th:main-helicity}. To keep the proofs and the statements simple, we decide to avoid this discussion. We recall that for a smooth solution $u$ of \eqref{Euler}, then $\omega = \curl(u)$ satisfies the Cauchy formula 
\begin{equation}
    \omega(x,t) = [ \nabla X_t \, \omega_0 ](X^{-1}_t(x)), \label{eq: cauchy formula}
\end{equation}
where $X_t$ is the flow map associated to the velocity field $u$ at time $t$. For the reader convenience, we enclose a proof of this standard formula in Lemma \ref{lem:computation-curl}. In the smooth setting, on a bounded domain $\Omega$, the boundary is invariant under the flow map thanks to the impermeability condition. Therefore, if $\omega_0 \equiv 0$ on $\partial \Omega$, then $\omega(t) \equiv 0$ on $\partial \Omega$ for any $t>0$. In this case, a simple integration by parts shows that the total helicity is conserved. 

\section{Tools} \label{s: tools}

In this section, we collect some basic tools that will be used throughout this note. 

\subsection{The fractional Sobolev space $H^{\sfrac{1}{2}}$} \label{ss: fractional sobolev space}

In order to keep this note self-contained, we recall some basic facts on the fractional Sobolev space $H^{\sfrac{1}{2}}$ and its dual. We follow \cite{GL23}*{Chapter 6}. Given an open set $\Omega \subset \R^d$, we denote by
\begin{equation}
    [u]_{H^{\sfrac{1}{2}  }(\Omega)}^2 \coloneqq \iint_{\Omega \times \Omega} \frac{\abs{u(x) -u(y)}^2}{\abs{x-y}^{d+1}}\, dx \, dy, \qquad \norm{u}_{H^{\sfrac{1}{2}  }(\Omega)} \coloneqq \norm{u}_{L^2(\Omega)} + [u]_{H^{\sfrac{1}{2}  }} 
\end{equation}
the Gagliardo seminorm and the fractional Sobolev norm. Similarly, we say that $f \in H^{\sfrac{1}{2}}_{\rm loc}(\Omega)$ if $f \in H^{\sfrac{1}{2}}(U)$ for any bounded open set $U$ with $\overline{U} \subset \Omega$. If $\Omega = \R^d$, then $H^{\sfrac{1}{2}}$ can be equivalently defined via the Fourier transform (see \cites{DNPV12, GL23}). Indeed, by \cite{DNPV12}*{Proposition 3.7}, there exists a purely dimensional constant $\gamma_d$ such that 
\begin{equation}
    [f]_{H^{\sfrac{1}{2}}(\R^d)}^2 = \gamma_d \int_{\R^d} \abs{\xi} \abs{\hat{f}(\xi)}^2 \, d \xi. 
\end{equation}
It is classical to define $H^{-\sfrac{1}{2}  } (\Omega) = (H^{\sfrac{1}{2}  }_0 (\Omega))'$, where $H^{\sfrac{1}{2}  }_0 (\Omega)$ is the completion of $C^\infty_c (\Omega)$ with respect to the $H^{\sfrac{1}{2}  }$ norm. However, if $\Omega$ is a bounded open set with Lipschitz boundary, by \cite{GL23}*{Theorem 6.78} we have $H^{\sfrac{1}{2}  } (\Omega) = H^{\sfrac{1}{2}  }_0 (\Omega)$ and therefore $H^{-\sfrac{1}{2}  } (\Omega) = (H^{\sfrac{1}{2}  } (\Omega) )'$. Given $U, V$ any open sets such that $U \subset \joinrel \subset V$, we denote by $H^{\sfrac{1}{2}}_U(V)$ the collection of functions in $H^{\sfrac{1}{2}}(V)$ with support in $U$. Since $H^{\sfrac{1}{2}}_U(V)$ is a closed linear subspace of $H^{\sfrac{1}{2}}(V)$, then $H^{\sfrac{1}{2}}_U(V)$ is a Banach space with the norm $\norm{\cdot}_{H^{\sfrac{1}{2}}(V)}$ and we denote by $H^{-\sfrac{1}{2}}_U(V)$ its dual. Unless otherwise specified, we denote by $\langle \cdot, \cdot \rangle$ the duality pairing between $H^{-\sfrac{1}{2}}_U(V)$ and $H^{\sfrac{1}{2}}_U(V)$. 

\begin{lem} \label{l: approx by convolution}
Let $\Omega \subset \R^d$ be any open set and $f \in H^{\sfrac{1}{2}  }_{\rm loc}(\Omega)$. Define $f_\e = \tilde{f} * \rho_\e$, where $\rho$ is a standard mollifier and $\tilde{f}$ is the extension of $f$ to $0$ outside $\Omega$. Then $f_\e \to f$ in $H^{\sfrac{1}{2}}_{\rm loc}(\Omega)$. Moreover, for any open sets $U, V$ with $U \subset \joinrel \subset V \subset \joinrel \subset \Omega$ and $\e < \e_0 = \text{dist} (U, V^c)$ it holds  
\begin{equation}
    \norm{f_\e}_{H^{\sfrac{1}{2}  }(U)} \leq \norm{f}_{H^{\sfrac{1}{2}  }(V)}. \label{eq: bound convolution H^1/2}
\end{equation}
\end{lem}

\begin{proof}
Fix $U, V$ open sets with $U \subset \joinrel \subset V \subset \joinrel \subset \Omega$ and let $\e < \e_0 = \text{dist} (U, V^c)$. For any $\alpha >0$, we set $D_\alpha : = \{ (x,y) \in \R^d \times \R^d \colon \abs{x-y} \leq \alpha \}$. Then, we estimate
\begin{align}
    \iint_{D_\alpha \cap (U \times U )} \frac{ \abs{f_\e(x) - f_\e(y)}^2}{\abs{x-y}^{d+1}} \, dy \, dx & \leq \iint_{D_\alpha \cap (U \times U)} \frac{ \abs{\int_{B_\e} \abs{f(x-z)-f(y-z)} \rho_\e(z) \, dz}^2}{\abs{x-y}^{d+1}}\, dy \, dx  
    \\ & \leq \iint_{D_\alpha \cap (U \times U)} \int_{B_\e} \frac{ \abs{f(x-z)-f(y-z)}^2 }{\abs{x-y}^{d+1}} \rho_\e(z) \, dz \, dy  \, dx 
    \\ & \leq \iint_{D_\alpha \cap (V \times V)} \frac{ \abs{f(x)-f(y)}^2 }{\abs{x-y}^{d+1}} \, dy \, dx. 
\end{align}
Therefore, we write 
\begin{align}
    [f-f_\e]_{H^{\sfrac{1}{2}  }(U)}^2 & = \left[ \iint_{D_\alpha \cap (U \times U)} + \iint_{D_\alpha^c \cap (U \times U)}  \right] \frac{\abs{f_\e(x) -f_\e(y) -f(x) +f(y)}^2}{\abs{x-y}^{d+1}}\, dx \, dy  
    \\ & \leq 2 \iint_{D_\alpha \cap (V \times V)} \frac{\abs{f(x)- f(y)}^2}{\abs{x-y}^{d+1}} \, dx \, dy + 2 \mathcal{L}^d(U) \alpha^{-1-d} \norm{f-f_\e}_{L^2(U)}^2. 
\end{align}
Letting $\e \to 0$ and recalling that $f_\e \to f$ in $L^2(U)$, we have that 
\begin{equation}
    \limsup_{\e \to 0} [f-f_\e]_{H^{\sfrac{1}{2}  }(U)}^2 \leq 2 \iint_{D_\alpha \cap (V \times V)} \frac{\abs{f(x)- f(y)}^2}{\abs{x-y}^{d+1}} \, dx \, dy. 
\end{equation}
Letting $\alpha \to 0$, the right hand side goes to $0$ by dominated convergence, since $f \in H^{\sfrac{1}{2}}(V)$. The proof of \eqref{eq: bound convolution H^1/2} is completely analogous. 
\end{proof}

\begin{lem} \label{l: extension lemma}
Let $U, V$ be bounded open sets such that $ U \subset \joinrel \subset V$. Let $f \in H^{\sfrac{1}{2}}(V)$ and let $\chi \in C^\infty_c(V; [0,1])$ be such that $\chi \equiv 1$ on $U$. Letting $g = \chi f$, then 
\begin{equation}
    \norm{g}_{H^{\sfrac{1}{2} }(\R^d)} \lesssim \norm{f}_{H^{\sfrac{1}{2}  }(V)}, \label{eq: extension norm}
\end{equation} 
where the implicit constant depends only on $d,V,U, \chi$. 
\end{lem}

\begin{proof}
It is enough to estimate the Gagliardo seminorm on $\R^d$. Splitting the double integral and recalling that $g$ has compact support in $V$, we compute 
\begin{align}
    [g]_{H^{\sfrac{1}{2}  }(\R^d)}^2 & = \iint_{U \times U} \frac{\abs{g(x)- g(y)}^2}{\abs{x-y}^{d+1}}\, dx \, dy + \iint_{(V \setminus U) \times (V \setminus U)} \frac{\abs{g(x)- g(y)}^2}{\abs{x-y}^{d+1}}\, dx \, dy 
    \\ & \hspace{1 cm}+ 2 \iint_{U \times (V \setminus U)} \frac{\abs{g(x)- g(y)}^2}{\abs{x-y}^{d+1}}\, dx \, dy +  2 \iint_{V \times V^c} \frac{\abs{g(x)- g(y)}^2}{\abs{x-y}^{d+1}}\, dx \, dy
    \\ & = I + II + III + IV. 
\end{align}
Since $\chi \equiv 1$ on $U$, we have that $I = [f]_{H^{\sfrac{1}{2}  }(U)}^2$. Since $\chi$ is Lipschitz continuous, we estimate 
\begin{align}
    \sup_{x \in \R^d} \abs{\int_{\R^d} \frac{\abs{\chi(x)- \chi(y)}^2}{\abs{x-y}^{d+1}} \, dy } \lesssim \sup_{x \in \R^d} \left[ \int_{B_1(x)} \frac{1}{\abs{x-y}^{d-1}} \, dy + \int_{B_1(x)^c} \frac{1}{\abs{x-y}^{d+1}} \, dy \right] < +\infty. 
\end{align}
Hence, we have  
\begin{align}
    II & \lesssim \iint_{(V \setminus U) \times (V \setminus U) } \abs{\chi(x)}^2 \frac{\abs{f(x) - f(y)}^2}{\abs{x-y}^{d+1}}\, dx \, dy + \iint_{(V \setminus U) \times (V \setminus U)} \abs{f(y)}^2 \frac{\abs{\chi(x)-\chi(y)}^2}{\abs{x-y}^{d+1}} \, dx \, dy 
    \\ & \lesssim [f]_{H^{\sfrac{1}{2}  }(V \setminus U)}^2 + \norm{f}_{L^2(V \setminus U)}^2. 
\end{align}
Similarly, we estimate 
\begin{align}
    III & \lesssim [f]_{H^{\sfrac{1}{2}  }(V)}^2 + \norm{f}_{L^2(V\setminus U)}^2. 
\end{align}
Lastly, recalling that $\chi$ has compact support in $V$, we estimate 
\begin{align}
    IV & \lesssim \iint_{V \times V^c} \abs{f(x)}^2 \frac{\abs{\chi(x)- \chi(y)}^2 }{\abs{x-y}^{d+1}}\, dx \, dy \lesssim \norm{f}_{L^2(V)}^2, 
\end{align} 
thus proving \eqref{eq: extension norm}.
\end{proof}

\begin{lem} \label{derivative of H^{1/2}}
Let $U,V$ be bounded open sets such that $U \subset \joinrel \subset V$ and let $f \in H^{\sfrac{1}{2}}(V)$. Then, $\partial_i f \in H^{-\sfrac{1}{2}}_U(V)$ for any $i = 1, \dots, d$. More precisely, we find an implicit constant depending only on $U,V$ such that 
\begin{equation}
    \abs{ \langle \partial_i f, \phi \rangle } \lesssim \norm{f}_{H^{\sfrac{1}{2}  }(V)} \norm{\phi}_{H^{\sfrac{1}{2}  }(V)} \qquad \forall \phi \in H^{\sfrac{1}{2}}_U(V). 
\end{equation}
\end{lem}

\begin{proof}
Fix bounded open sets $W,Z$ such that $ U \subset \joinrel \subset W \subset \joinrel \subset Z \subset \joinrel \subset V$ and let $\chi \in C^\infty_c(V; [0,1])$ such that $\chi \equiv 1$ on $W$. Fix $i \in \{1, \dots, d\}$. From now on, we neglect multiplicative constants depending on $d, U,V, W, Z$. Letting $g = \chi f$, by \cref{l: extension lemma} it results that $g \in H^{\sfrac{1}{2}}(\R^d)$. Then, by the properties of Fourier transform, for any $\phi \in C^\infty_c(\R^d)$ we have 
\begin{align}
    \abs{\langle \partial_i g, \phi \rangle } & \lesssim \int_{\R^d} \abs{\xi_i} \abs{\hat{g}(\xi)} \abs{\hat\phi(\xi)} \, d \xi \lesssim \norm{\abs{\xi}^{\sfrac{1}{2}} \hat{g} }_{L^2(\R^d)} \norm{\abs{\xi}^{\sfrac{1}{2}} \hat{\phi} }_{L^2(\R^d)}\lesssim \norm{g}_{H^{\sfrac{1}{2}}(\R^d)} \norm{\phi}_{H^{\sfrac{1}{2}} (\R^d)}, 
\end{align} 
where we use the characterization of fractional Sobolev norms on $\R^d$ with the Fourier transform (see e.g. \cite{DNPV12}*{Proposition 3.7}). Hence, by \eqref{eq: extension norm}, we infer that  
\begin{equation}
    \abs{ \langle \partial_i f, \phi \rangle } \lesssim \norm{f}_{H^{\sfrac{1}{2}  }(Z)} \norm{\phi}_{H^{\sfrac{1}{2}  }(Z)} \qquad \forall \phi \in C^\infty_c(W).   
\end{equation} 
To conclude, we prove that the estimate above holds for any test function $\phi \in H^{\sfrac{1}{2}}_U(V)$. Indeed, given $\phi \in H^{\sfrac{1}{2}}_U(V)$ and letting $\phi_\e$ be the approximation by convolution, by \cref{l: approx by convolution} $\phi_\e \to \phi$ in $H^{\sfrac{1}{2}}(Z)$ and $\phi_\e$ has support in $W$ for $\e$ small enough. 
\end{proof}

\begin{lem} \label{l: convergence in H^-1/2}
Let $U,V$ be bounded open sets such that $U \subset \joinrel \subset V$ and let $f \in H^{\sfrac{1}{2}}(V)$. Let $f_\e$ be defined as in \cref{l: approx by convolution}. Let $\{\phi^\e\}_\e$ be a sequence in $H^{\sfrac{1}{2}}_U(V)$ such that $\phi^\e\to \phi$ in $H^{\sfrac{1}{2}  }(V)$. Then, for any $i = 1, \dots, d$, it holds that 
\begin{equation}
    \lim_{\e \to 0} \langle \partial_i f_\e, \phi^\e\rangle = \langle \partial_i f, \phi \rangle. 
\end{equation}
\end{lem}

\begin{proof}
Fix bounded open sets $W,Z$ such that $ U \subset \joinrel \subset W \subset \joinrel \subset Z \subset \joinrel \subset V$ and define 
$$\e_0 = \min\{ \text{dist}(U, W^c), \text{dist} (W, Z^c), \text{dist}(Z, V^c)\}. $$
We neglect constants depending only on $U,V, Z, W,d$. Hence, for any $\e \leq \e_0$, since $\phi_\e = \tilde{\phi} * \rho_\e$ has support in $W$ ($\tilde{\phi}$ is the extension of $\phi$ to $0$ outside $V$), by \cref{l: approx by convolution} we have 
\begin{align} 
    \abs{\langle \partial_i f_\e, \phi^\e \rangle - \langle \partial_i f, \phi\rangle } & \leq \abs{\langle \partial_i f_\e - \partial_i f, \phi \rangle} + \abs{\langle \partial_i f_\e, \phi- \phi^\e \rangle} 
    \\ & \lesssim \abs{\langle \partial_i f, \phi- \phi_\e\rangle} + \norm{f_\e}_{H^{\sfrac{1}{2}  }(Z)} \norm{\phi- \phi_\e}_{H^{\sfrac{1}{2}  }(Z)}
    \\ & \lesssim \norm{f}_{H^{\sfrac{1}{2}  }(V)} \norm{\phi- \phi_\e}_{H^{\sfrac{1}{2}  }(V)}.
\end{align} 
\end{proof}

\begin{lem} \label{l: time dependent duality}
Let $U,V$ be bounded open sets such that $U \subset \joinrel \subset V$, let $I$ be a time interval and let $f \in L^2(I; H^{\sfrac{1}{2}}(V))$. Fix $i \in \{1, \dots,d\}$. Then, $\partial_i f \in (L^2(I; H^{\sfrac{1}{2}}_U(V)))'$ and there exists an implicit constant depending only on $U,V$ such that 
$$ \norm{\partial_i f}_{(L^2(I; H^{\sfrac{1}{2}}_U(V)))'} \lesssim \norm{f}_{L^2(I; H^{\sfrac{1}{2}}(V))}.$$
\end{lem}

\begin{proof}
Let $f_\e, \phi_\e$ be the spatial mollification of $f,\phi$ respectively. For any $i = 1, \dots, d$, by \cref{l: convergence in H^-1/2} it holds that 
\begin{equation}
    \lim_{\e \to 0} \langle \partial_i f_\e(t), \phi_\e(t) \rangle = \langle \partial_i f(t), \phi(t) \rangle
\end{equation}
for a.e. $t \in I$. Since the functions $t \mapsto \langle \partial_i f_\e(t), \phi_\e(t)\rangle$ are measurable for any $\e$ by Fubini's theorem, we infer that the same holds for $t \mapsto \langle \partial_i f(t), \phi(t)\rangle$. Then, by the estimates of \cref{l: approx by convolution}, \cref{derivative of H^{1/2}} and by dominated convergence, we infer that 
\begin{equation}
    \lim_{\e \to 0} \int_I \langle \partial_i f_\e(t), \phi_\e(t) \rangle \, dt = \int_I \langle \partial_i f(t), \phi(t) \rangle \, dt, 
\end{equation}
which defines the duality relation. Moreover, we have that 
\begin{equation}
    \abs{  \int_I \langle \partial_i f(t), \phi(t) \rangle \, dt } \lesssim \norm{f}_{L^2(I; H^{\sfrac{1}{2}}(V))} \norm{\phi}_{L^2(I; H^{\sfrac{1}{2}}(V))}. 
\end{equation}
\end{proof}

Thanks to \cref{l: time dependent duality}, given $f, g \in L^2_t H^{\sfrac{1}{2}}_x$ we can define $f  \partial_i g$ in $\mathcal{D}'(\Omega \times (0,T))$. 

\begin{Def} \label{coupling in H^1/2 - H^-1/2}
Let $\Omega \subset \R^3$ be any open set and let $f,g \in L^2_{\rm loc}((0,T); H^{\sfrac{1}{2}}_{\rm loc}(\Omega))$. We define the product $f \partial_i g$ in $\mathcal{D}'(\Omega \times (0,T))$ by \begin{equation}
\langle f  \partial_i g , \phi \rangle = \int_0^T \langle \partial_i g(t), f(t)\phi(t) \rangle  \, dt \qquad \forall \phi \in C^\infty_c(\Omega \times (0,T)). \label{helicity equation}
\end{equation}
\end{Def}

\begin{lem} \label{H12 is algebra}
Let $\Omega$ be any open set and let $f,g \in H^{\sfrac{1}{2}  } \cap L^\infty(\Omega)$. Then, $fg \in H^{\sfrac{1}{2}  }\cap L^\infty(\Omega)$ and, up to a universal constant, it holds  
\begin{equation}
    \norm{fg}_{H^{\sfrac{1}{2}  }(\Omega)} \lesssim \norm{f}_{L^\infty(\Omega)} \norm{g}_{L^2(\Omega)} + \norm{g}_{L^\infty(\Omega)} [f]_{H^{\sfrac{1}{2}  }(\Omega)} + \norm{f}_{L^\infty(\Omega)} [g]_{H^{\sfrac{1}{2}  }(\Omega)}. 
\end{equation}
\end{lem}

\begin{proof}
By direct computation, we have that 
\begin{align}
\norm{fg}_{L^2(\Omega)} & \leq \norm{f}_{L^\infty(\Omega)} \norm{g}_{L^2(\Omega)},
\\ [fg]_{H^{\sfrac{1}{2}  }(\Omega)}^2 & = \iint_{\Omega \times \Omega} \frac{\abs{f(x) g(x)- f(y) g(y)}^2}{\abs{x-y}^{d+1}} \, dx \, dy 
\\ & \lesssim \iint_{\Omega \times \Omega} \frac{\abs{g(x)}^2 \abs{f(x) - f(y) }^2}{\abs{x-y}^{d+1}} \, dx \, dy + \iint_{ \Omega \times \Omega} \frac{\abs{f(y)}^2 \abs{ g(x)-  g(y)} ^2}{\abs{x-y}^{d+1}} \, dx \, dy 
\\ & \lesssim \norm{g}_{L^\infty(\Omega)}^2 [f]_{H^{\sfrac{1}{2}  }(\Omega)}^2 + \norm{f}_{L^\infty(\Omega)}^2 \norm{g}_{H^{\sfrac{1}{2}  }(\Omega)}^2. 
\end{align}
\end{proof}

\begin{lem} \label{convergence of the squares}
Let $\Omega$ be any open set, $f,g \in H^{\sfrac{1}{2}  }_{\rm loc} \cap L^\infty_{\rm loc}(\Omega)$ and let $\rho$ be a standard mollifier. Define $f_\e = \tilde{f}* \rho_\e, g_\e = \tilde{g}* \rho_\e$, where $\tilde{f}, \tilde{g}$ are the extensions of $f,g$ to $0$ outside $\Omega$. Then $f_\e g_\e \to f g$ in $H^{\sfrac{1}{2}  }_{\rm loc}(\Omega)$. 
\end{lem}

\begin{proof} 
Fix bounded open sets $U, V$ such that $ U \subset \joinrel \subset V \subset \joinrel \subset \Omega$ and let $\e \leq \e_0 = \text{dist}(U, V^c)$. Fix $\alpha >0$ and set $D_\alpha : = \{ (x,y) \in \R^d \times \R^d \colon \abs{x-y} \leq \alpha \}$. We neglect constants depending on $d, U,V$. Then, for any $\e \leq \e_0$, since $\norm{f_\e}_{L^\infty(U)} \leq \norm{f}_{L^\infty(V)}$ and the same holds for $g_\e,g$, we estimate
\begin{align}
    & \iint_{D_\alpha \cap (U \times U) }  \frac{ \abs{(f_\e g_\e)(x) - (f_\e g_\e)(y)}^2}{\abs{x-y}^{d+1}} \, dy \, dx 
    \\ & \hspace{0.5 cm} \lesssim \iint_{D_\alpha \cap  (U \times U) } \abs{g_\e(x)}^2 \frac{ \abs{f_\e(x) - f_\e(y)}^2}{\abs{x-y}^{d+1}} \, dy \, dx + \iint_{D_\alpha \cap ( U \times U )} \abs{f_\e(y)}^2 \frac{ \abs{g_\e(x) - g_\e(y)}^2}{\abs{x-y}^{d+1}} \, dy \, dx
    \\ & \hspace{0.5 cm} \lesssim \norm{g}_{L^\infty(V)}^2 \iint_{D_\alpha \cap ( V \times V )} \frac{ \abs{f(x)-f(y)}^2 }{\abs{x-y}^{d+1}} \, dy \, dx + \norm{f}_{L^\infty(V)}^2 \iint_{D_\alpha \cap (V \times V)} \frac{ \abs{g(x)-g(y)}^2 }{\abs{x-y}^{d+1}} \, dy \, dx, 
\end{align}
as in the proof of \cref{l: approx by convolution}. Then, we write 
\begin{align}
     [(f_\e g_\e)-f g]_{H^{\sfrac{1}{2}  }(U)}^2 & = \left[ \iint_{D_\alpha \cap ( U \times U)} + \iint_{D_\alpha^c \cap ( U \times U )} \right] \frac{\abs{(f_\e g_\e)(x) -(f_\e g_\e)(y) -(fg)(x) +(fg)(y)}^2}{\abs{x-y}^{d+1}}\, dx \, dy  
    \\ & \lesssim \iint_{D_\alpha \cap (V \times V)} \frac{ \abs{f(x)-f(y)}^2 }{\abs{x-y}^{d+1}} \, dy \, dx +  \iint_{D_\alpha \cap ( V \times V)} \frac{ \abs{g(x)-g(y)}^2 }{\abs{x-y}^{d+1}} \, dy \, dx 
    \\ & \hspace{ 1 cm} + \alpha^{-1-d} \norm{fg-f_\e g_\e}_{L^2(U)}^2. 
\end{align}
Letting $\e \to 0$ and recalling that $f_\e g_\e \to f g$ in $L^2(U)$, we have that 
\begin{equation}
    \limsup_{\e \to 0} [f-f_\e]_{H^{\sfrac{1}{2}  }(U)}^2 \lesssim \iint_{D_\alpha \cap ( V \times V)} \frac{\abs{f(x)- f(y)}^2}{\abs{x-y}^{d+1}} \, dx \, dy + \iint_{D_\alpha \cap ( V \times V)} \frac{ \abs{g(x)-g(y)}^2 }{\abs{x-y}^{d+1}} \, dy \, dx. 
\end{equation}
Letting $\alpha \to 0$, then the right hand side goes to $0$ by dominated convergence, since $f,g \in H^{\sfrac{1}{2}  }(V)$. 
\end{proof}

\subsection{Some commutator estimates of Besov vector fields} \label{ss: commutator estimates and besov}

We recall the definition of Besov spaces in arbitrary open sets (see the monograph \cite{T83} for an extensive presentation of the topic) and we provide some standard mollification estimates in the spirit of \cites{CET94, DR19}. We include the short proof for the reader's convenience. Let $O \subset \joinrel \subset \Omega \subset \R^d$ be open sets, $\theta \in (0,1)$ and $p \in [1, +\infty]$. Given $u \in L^p(O)$ , we say that $u \in B^\theta_{p, \infty}(O)$ if it holds that
\begin{equation} \label{eq: besov}
    \sup_{\abs{h} \leq \text{dist}(O, \Omega^c)} \frac{\norm{u(\cdot + h) - u(\cdot)}_{L^p(O)}}{\abs{h}^\theta} =: [u]_{B^\theta_{p, \infty}(O)} < + \infty. 
\end{equation} 
We denote by 
$$\norm{u}_{B^\theta_{p,\infty}(O)} := [u]_{B^\theta_{p,\infty}(O)} + \norm{u}_{L^p(O)}. $$
We say that $u \in B^\theta_{p, c_0}(O)$ if $u \in L^p(O)$ and there exists a modulus of continuity $\ell_{u, O}$ such that for any $\e \leq \text{dist}(O, \Omega^c)$ it holds that 
\begin{equation} \label{eq: little besov}
    \sup_{\abs{h} \leq \e} \frac{\norm{u(\cdot + h) - u(\cdot)}_{L^p(O)}}{\abs{h}^\theta} \leq \ell_{u, O}(\e). 
\end{equation}
We say that $u \in L^q(I; B^\theta_{p, c_0}(O))$ if $u \in L^q(I; B^\theta_{p, \infty}(O))$ and $u(t) \in B^\theta_{p, c_0}(O)$ for a.e. $t \in I$. 

\begin{lem} \label{l: commutator estimates}
Let $O \subset \joinrel \subset  \Omega \subset \R^d$ be open sets and let $f,g \in B^\theta_{p, c_0}(O)$. Let $\rho$ be a standard mollifier and for any $\e \leq \e_0 = \text{dist}(U, \Omega^c)$ let $f_\e = f*\rho_\e, g_\e = g* \rho_\e$. There are implicit constants independent on $f,g,\e,k$ such that for any $\e \leq \e_0$ it holds  
\begin{equation} \label{eq: besov gradient}
    \norm{\nabla^k f_\e}_{L^p(U)} \lesssim \e^{-k + \theta} \norm{f}_{B^\theta_{p, \infty} (V) } \ell_{f,U}(\e) \qquad \forall k \geq 1,  
\end{equation}
\begin{equation} \label{eq: besov quadratic}
    \norm{f_\e g_\e - (fg)_\e}_{L^{\sfrac{p}{2}}(U)} \lesssim \e^{2\theta} \norm{f}_{B^\theta_{p, \infty}(V)} \norm{g}_{B^\theta_{p, \infty}} \ell_{f,U}(\e) \ell_{g, U}(\e) \qquad \text{ if } p \geq 2. 
\end{equation}
\end{lem} 

\begin{proof} 
We assume $p < \infty, k \geq 1$ and we neglect constants depending on $k,p, \rho$. The case $p = \infty$ is completely analogous. Following \cite{DR19}*{Proposition 2.1}, given $k \geq 1$ and $\e \leq \e_0$ we have that 
\begin{equation}
    \nabla^k f_\e(x) = \e^{-k} \int_{B_1} (f(x-\e y) - f(x)) \nabla^k \rho(y) \, dy.   
\end{equation}
Then, using Jensen's inequality, we infer that 
\begin{align}
    \int_{U} \abs{\nabla f_\e(x)}^p \, dx & \lesssim \e^{-kp} \int_{B_1} \int_U \abs{f(x-\e y)-f(x)}^p \, dx \abs{\nabla^k \rho(y)} \, dy \lesssim \e^{p (\theta-k)} [f]_{B^\theta_{p, \infty}(O)}^p \ell_{f, O}(\e)^p. 
\end{align}
Then, \eqref{eq: besov gradient} follows by taking the $p$ root. Similarly, to prove \eqref{eq: besov quadratic}, if $p \geq 2$ we write 
\begin{align}
    (f g)_\e(x) - f_\e(x) g_\e(x) & = \int_{B_1} (f(x-\e y) - f(x)) (g(x-\e y) - g(x)) \rho(y) \, dy
    \\ & - \left( \int_{B_1} (f(x-\e y) - f(x) ) \rho(y) \, dy \right) \left( \int_{B_1} (g(x-\e y) - g(x)) \rho(y)\, dy \right) 
    \\ & = I_\e(x) +II_\e(x).  
\end{align}
Then, by Jensen's and H\"older's inequality, we infer that 
\begin{align}
    \int_U \abs{ I_\e(x)}^{\sfrac{p}{2}} \, dx & \lesssim \int_{B_1} \int_U \abs{f(x- \e y)- f(x)}^{\sfrac{p}{2}} \abs{g(x-\e y) - g(x)}^{\sfrac{p}{2}} \, dx \, \rho(y) \, dy 
    \\ & \hspace{1 cm} 
    \\ & \lesssim \int_{B_1} \left( \int_U \abs{f(x-\e y) - f(x)}^p \, dx \right)^{\sfrac{1}{2}} \left( \int_U \abs{g(x-\e y) - g(x)}^p \right)^{\sfrac{1}{2}} \rho(y) \, dy  
    \\ & \lesssim \e^{\theta p} \left[ \norm{f}_{B^\theta_{p, \infty}(U)} \norm{g}_{B^\theta_{p,\infty}(U)} \ell_{f,U}(\e) \ell_{g,U}(\e) \right]^{\sfrac{p}{2}}.   
\end{align}
Similarly, we estimate 
\begin{equation}
    \int_U \abs{II_\e(x)}^{\sfrac{p}{2}} \, dx \lesssim \e^{\theta p} \left[ \norm{f}_{B^\theta_{p, \infty}(U)} \norm{g}_{B^\theta_{p,\infty}(U)} \ell_{f,U}(\e) \ell_{g,U}(\e) \right]^{\sfrac{p}{2}}.
\end{equation}
Then, \eqref{eq: besov gradient} follows by taking the $\sfrac{p}{2}$ root. 
\end{proof}

\subsection{Boundary trace} \label{ss: boundary trace} We recall the distributional notion of normal trace of a measure divergence vector field. Given an open set $\Omega \subset \R^d$, $p \in [1, +\infty]$ and $u \in L^p(\Omega)$ a vector field whose divergence is a Radon measure $\lambda$ on $\Omega$, the \emph{outer distributional normal trace} $\Tr_n(u, \partial \Omega)$ is defined in $\mathcal{D}'(\R^d)$ by 
\begin{equation}
    \langle \Tr_n(u, \partial \Omega); \phi \rangle := \int_{\Omega} \phi \, d \lambda + \int_{\Omega} u \cdot \nabla \phi \, dx \qquad \forall \phi \in C^\infty_c(\R^d). 
\end{equation}

As described in \cite{ACM05}, it is known that if $u \in L^\infty(\Omega)$ and $u$ has measure divergence, then the outer distributional normal trace is represented by a function in $L^\infty(\partial \Omega; \mathcal{H}^{d-1})$. Since this notion of trace is too weak to be handled in many situations, we recall the definition of the \emph{(normal) Lebesgue boundary trace} and some basic properties studied by the first author in \cites{DRINV23, CDRIN24}. This notion of trace is appropriate to deal with nonlinear problems and it lies between the distributional one for measure-divergence vector fields \cite{ACM05} and the strong one for BV functions \cite{AFP00}. 

\begin{Def} \label{def:normal}
Let $\Omega \subset \R^d$ be a bounded open set with Lipschitz boundary and let $u \in L^\infty (\Omega)$. We say that $u$ admits \emph{full Lebesgue trace} $u_{\partial \Omega} \in L^\infty(\partial \Omega; \R^d)$ on $\partial \Omega$ if for any sequence $r_k \to 0$ it holds that 
\begin{equation}
    \lim_{r_k \to 0} \fint_{B_{r_k}(x) \cap \Omega} \abs{u(y) - u_{\partial \Omega}(x)}\, dy = 0 \qquad \text{for $\mathcal{H}^{d-1}$-a.e. } x \in \partial \Omega. 
\end{equation}
If $u \in L^\infty(\Omega; \R^d)$, we say that $u$ admits \emph{outer normal Lebesgue trace} $u_n \in L^\infty(\partial \Omega)$ on $\partial \Omega$ if for any sequence $r_k \to 0 $ it holds that
\begin{equation}
    \lim_{r_k \to 0} \fint_{ B_{r_k} (x) \cap \Omega} \abs{ (u \cdot \nabla d_{\partial \Omega} ) (y) + u_n(x)} \, dy = 0  \qquad \text{for $\mathcal{H}^{d-1}$-a.e. } x \in \partial \Omega. 
\end{equation}
\end{Def}

Let $\Omega \subset \R^d$ be a bounded open set with Lipschitz boundary and let $n: \partial \Omega \to \mathbb{S}^{d-1}$ be the outer unit normal. Some remarks are in order. 

\begin{enumerate}
    \item [(i)] By \cite{DRINV23}*{Theorem 2.4} for $\mathcal{H}^{d-1}$-a.e. $x \in \partial \Omega$ we have that 
\begin{equation}
    \lim_{r \to 0} \fint_{B_r(x)} \abs{\nabla d_{\partial \Omega}(y) + n(x)} \, dy = 0. 
\end{equation}
Hence, if $u$ has a full trace $u_{\partial \Omega}$ in the sense of \cref{def:normal}, then $u$ has outer normal Lebesgue trace and it holds that $u_n = u_{\partial \Omega} \cdot n$.  
\item [(ii)] If $u$ has full trace (normal Lebesgue trace) $0$ at $\partial \Omega$, then $gu$ has full Lebesgue trace (respectively, normal Lebesgue trace) $0$ at $\partial \Omega$ for any $g \in L^\infty(\Omega)$.
\item [(iii)] If $f,g$ are bounded scalar functions with full trace at $\partial \Omega$, then $fg$ has full Lebesgue trace and it holds $(fg)_{\partial \Omega} = f_{\partial \Omega} g_{\partial \Omega}$. 
\item [(iv)] If $u$ is a bounded vector field with normal Lebesgue trace at $\partial \Omega$ and $f$ is a bounded scalar function with full trace $f_{\partial \Omega}$ at $\partial \Omega$, then $fu$ has normal Lebesgue trace and $(fu)_n = f_{\partial \Omega} u_n$. 
\end{enumerate}

We recall one of the main properties of the normal Lebesgue trace \cite{CDRIN24}*{Corollary 3.3}.
     
\begin{prop} \label{prop:boundary}
    Let $\Omega \subset \R^d$ be a bounded open set with Lipschitz boundary and let $u \in L^\infty(\Omega)$ be a vector field with outer normal Lebesgue trace $u_n$ according to \cref{def:normal}. Setting $\chi_r (x) : = \min (r^{-1}\text{dist}_{\partial \Omega} (x), 1 )$, it holds that   
    $$ \lim_{r \to 0 } \int_{\Omega} \phi U \cdot \nabla \chi_r \, dx = - \int_{\partial \Omega} \phi U_n \, d \mathcal{H}^{d-1} \qquad \forall \phi \in C^\infty_c(\R^d). $$
\end{prop}

\subsection{Explicit formula for the vorticity}

For the reader convenience, we discuss the proof of the Cauchy formula for the vorticity of a smooth Euler flow \eqref{eq: cauchy formula}. 

\begin{lem} \label{lem:computation-curl}
Let $\Omega \subset \R^3$ be a smooth bounded open set open bounded and let $u \in C^\infty ([0,T] \times \overline{\Omega} ; \R^3)$ be a smooth solution to \eqref{Euler} on $\Omega$ with the boundary condition $u(t) \cdot n \equiv 0$ at $\partial \Omega$ for any $t \in (0,T)$ and smooth initial datum $u_0$. Letting $X_t$ be the flow map associated to the velocity field $u$ at time $t$, then $\omega = \curl (u)$ satisfies
\begin{equation} \label{eq: formula vorticity}
    \omega (x,t ) = [ \nabla X_t \, \omega_0 ] (X^{-1}_t (x)), 
\end{equation}

\end{lem}

\begin{proof}
By Gr\"onwall lemma, it is clear that given a smooth divergence free velocity field $u : [0,T] \times \overline{\Omega} \to \R^3$, there exists at most a unique smooth solution $\omega$ to \eqref{euler vorticity} with a smooth initial datum $\omega_0 = \curl (u_0)$. Then, it is enough to check that $\omega$ given by \eqref{eq: formula vorticity} satisfies \eqref{euler vorticity}. Since, $u$ is divergence free and it satisfies $u(t) \cdot n \equiv 0$ at $\partial \Omega$ for any $t \in (0,T)$, then $X_t$ is measure preserving, $X_t (\Omega) = \Omega$ and
$$\int_{\Omega} \varphi (x) \omega (x,t) \, dx = \int_{\Omega} \varphi (X_t (x)) \nabla X_t (x) \omega_0 (x) \, dx \qquad \forall \phi \in C^\infty_c(\Omega). $$
Therefore, for any $\varphi \in C^\infty_c (\Omega)$ we have
     \begin{align}
         \int_{\Omega} \phi (x) \partial_t \omega (x,t) \, dx &  = \frac{d}{dt} \int_{\Omega} \phi (X_t (x)) \nabla X_t (x) \omega_0 (x) \, dx 
         \\ & =  \int_{\Omega} (\nabla \phi (X_t ) \cdot \partial_t X_t)  \nabla X_t  \omega_0  \, dx  + \int_{\Omega} \phi (X_t ) (\nabla \partial_t X_t )  \omega_0  \, dx 
         \\ & =  \int_{\Omega} ( \nabla \phi (X_t ) \cdot u (X_t) ) \nabla X_t \omega_0 \, dx + \int_{\Omega} \phi (X_t) \nabla u (X_t) \nabla X_t  \omega_0 \,  dx
         \\ & = \int_{\Omega} (  \nabla \phi (x) \cdot u(x,t) ) \omega (x,t) \, dx + \int_{\Omega} \phi(x) \nabla u(x,t) \omega(x,t) \, dx
         \\
         & = \int_{\Omega} ( \omega(x,t) \cdot \nabla u (x,t) - u (x,t) \cdot \nabla \omega (x,t) ) \phi (x) \, dx,
      \end{align}
      from which we conclude the proof. 
\end{proof}

\section{Proofs}

\subsection{Local helicity balance}
We discuss the proof of \cref{th:main-helicity}. 

\begin{proof}[Proof of \cref{th:main-helicity}] For the reader's convenience, we split the proof in several steps. 

\textsc{\underline{Step 1}}: To begin with, we discuss the regularity of the pressure. We prove that $p \in L^1_{\rm loc}((0,T); H^{\sfrac{1}{2}}_{\rm loc}(\Omega)) \cap L^q_{\rm loc}(\Omega \times (0,T))$. Fix bounded open sets $U \subset \joinrel \subset V \subset \joinrel \subset Z \subset \joinrel \subset \Omega, I \subset \joinrel \subset (0,T)$ and $\chi \in C^\infty_c(Z; [0,1])$ such that $\chi \equiv 1$ on $V$. We neglect multiplicative constants depending only on $U, V, Z$. Let $\tilde{u} = u \chi \in L^\infty(I \times \R^3) \cap L^2(I; H^{\sfrac{1}{2}}(\R^3))$ and let $\tilde{p}(t)$ be the unique solution to the elliptic problem 
\begin{equation}
    (-\Delta) \tilde{p}(t) = \divergence \divergence (\tilde{u}(t) \otimes \tilde{u}(t)) \qquad \text{in } \R^3, \label{eq: pressure equation}
\end{equation}
decaying at infinity. By standard Calderón--Zygmund estimates, it turns out that $\tilde{p}(t)$ satisfies 
\begin{equation} \label{eq: pressure 1}
    \norm{\tilde{p}(t)}_{H^{\sfrac{1}{2}}(\R^3)} \lesssim \norm{\tilde{u}(t) \otimes \tilde{u}(t)}_{H^{\sfrac{1}{2}}(\R^3)} \lesssim \norm{u(t)}_{H^{\sfrac{1}{2}}(Z)} \norm{u(t)}_{L^\infty(Z)},   
\end{equation} 
where we have used \cref{l: extension lemma} and \cref{H12 is algebra}. Similarly, for any $q \in (1, +\infty)$ we have that 
\begin{equation} \label{eq: pressure 2}
    \norm{\tilde{p}(t)}_{L^q(\R^3)} \lesssim \norm{\tilde{u}(t)}_{L^{2q}(\R^3)}^2 \lesssim \norm{u(t)}_{L^\infty(Z)}^2, 
\end{equation}
where the implicit constant depends also on $q$. Then, it is immediate to check that $p(t)- \tilde{p}(t)$ is harmonic in $V$ and by the mean value property, for any $k \geq 0$, we can estimate 
\begin{equation}
    \norm{p(t)- \tilde{p}(t)}_{C^k(U)} \lesssim \norm{p(t) - \tilde{p}(t)}_{L^1(V)} \lesssim \norm{p(t)}_{L^1(V)} + \norm{\tilde{p}(t)}_{L^1(V)} 
\end{equation}
where the implicit constant depends also on $k$. Then, by \eqref{eq: pressure 1} and \eqref{eq: pressure 2} we compute 
\begin{align}
    \norm{p}_{L^1(I; H^{\sfrac{1}{2}}(U)) } & \lesssim \norm{\tilde{p} - p}_{L^1(I; H^{\sfrac{1}{2}}(U))} + \norm{\tilde{p}}_{L^1(I;H^{\sfrac{1}{2}}(U))}
    \\ & \lesssim \norm{p}_{L^1(V \times I)} + \norm{\tilde{p}}_{L^1(V \times I)} + \norm{u}_{L^1(I;H^{\sfrac{1}{2}}(Z))} \norm{u}_{L^\infty( Z \times I)}
    \\ & \lesssim \norm{p}_{L^1(V \times I)} + \norm{u}_{L^\infty(Z\times I)}^2 + \norm{u}_{L^1(I;H^{\sfrac{1}{2}}(Z))} \norm{u}_{L^\infty( Z \times I)}. 
\end{align}
With the same argument, we prove that $p \in L^q_{\rm loc}(\Omega \times (0,T))$ for any $q < +\infty$. 

\textsc{\underline{Step 2}}: From now on, we fix bounded open sets $U, V, I$ as above and we set $\e \leq \e_0 = \text{dist}(U, V^c)$. Letting $u_\e, \omega_\e$ be the spatial mollification of $u, \omega$ in $U$, we check that $u_\e, \omega_\e \in W^{1,1}(U \times I)$ so that all the computations in the following steps are fully justified. Mollifying \eqref{Euler} in $U$ with respect to the spatial variable and taking the curl, we get 
\begin{equation} \label{eq: mollified equation 1}
    \partial_t u_\e + \divergence(u \otimes u)_\e + \nabla p_\e = 0, \qquad \partial_t \omega_\e + \curl(\divergence(u \otimes u)_\e) = 0 \qquad \text{in } U. 
\end{equation} 
From now on, we neglect multiplicative constants depending on the convolution kernel and on $U,V, I, k$. We estimate
$$\norm{\nabla^k u_\e}_{L^\infty(U\times I)} \lesssim \e^{-k} \norm{u}_{L^\infty(V \times I)}, \qquad \norm{\nabla^k(u \otimes u)_\e }_{L^\infty(U\times I)} \lesssim \e^{-k} \norm{u}_{L^\infty(V \times I)}^2, $$
$$\norm{\nabla p_\e}_{L^1(U \times I)} \lesssim \e^{-1} \norm{p}_{L^1(V \times I)}. $$
Thus, recalling \eqref{eq: mollified equation 1}, we infer that $\partial_t u_\e, \partial_t \omega_\e \in L^1(U \times I)$. 

\textsc{\underline{Step 3}}: The computations below hold in $U \times I$. Letting $R_\e = u_\e \otimes u_\e -(u\otimes u)_\e$, we write 
\begin{equation} \label{eq: mollified equation 2}
    \partial_t u_\e + \divergence(u_\e \otimes u_\e) + \nabla p_\e = \divergence(R_\e), 
\end{equation}
\begin{equation} \label{eq: mollified equation 3}
    \partial_t \omega_\e + (u_\e \cdot \nabla ) \omega_\e - (\omega_\e \cdot \nabla) u_\e = \curl \divergence(R_\e). 
\end{equation} 
Then, multiplying \eqref{eq: mollified equation 2} by $\omega_\e$ and \eqref{eq: mollified equation 3} by $u_\e$ and after some standard manipulations, we find  
\begin{equation}
    \begin{split}
        \partial_t(u_\e \cdot \omega_\e) + \divergence \bigg(u_\e (u_\e \cdot \omega_\e) + \bigg(p_\e - \frac{\abs{u_\e}^2}{2} \bigg) \omega_\e \bigg) & = \omega_\e \cdot \divergence R_\e + u_\e \cdot \curl \divergence(R_\e) 
        \\ & = \divergence (\omega_\e R_\e) - \nabla \omega_\e \colon R_\e + u_\e\cdot \curl \divergence R_\e.
    \end{split}
\end{equation}
Using the Levi--Civita notation, we check that 
\begin{equation}
    \begin{split}
        u_\e \cdot \curl \divergence R_\e & = u_\e^i \e_{i,j,k} \partial_j [\divergence R_\e]^k = \partial_j (u_\e^i \e_{i,j,k} [\divergence R_\e]^k) - \e_{i,j,k} \partial_j u_\e^i [\divergence R_\e]^k 
        \\ & = - \partial_j (\e_{j,i,k} u_\e^i [\divergence R_\e]^k) + \e_{k,j,i} \partial_j u_\e^i [\divergence R_\e]^k 
        \\ & = - \divergence (u_\e \times \divergence R_\e) + \curl u_\e \cdot \divergence R_\e. 
    \end{split}
\end{equation} 
We recall that 
$$\curl u_\e \cdot \divergence R_\e = \divergence(\omega_\e R_\e) - \nabla \omega_\e : R_\e. $$
Moreover, we have  
\begin{equation}
    \begin{split}
        \divergence (u_\e \times \divergence R_\e) & = \partial_i \e_{i,j,k} u_\e^j [\divergence R_\e]^k  = \partial_i ( \e_{i,j,k} \partial_l (u_\e^j R_\e^{k,l}) - \e_{i,j,k} (\partial_l u_\e^j R_\e^{k,l})) 
        \\ & = \partial_i ( \partial_l [u_\e \times R_\e^{\cdot, l}]^i ) - \partial_i (\partial_l u_\e \times R_\e^{\cdot, l})^i = \partial_l \divergence(u_\e \times R_\e^{\cdot, l}) - \divergence(\partial_l u_\e \times R_\e^{\cdot, l}),  
    \end{split}
\end{equation}
where we denote by $R^{\cdot, l}_\e$ the $l$-th column of the matrix $R_\e$. To summarize, we have 
\begin{equation} \label{eq:main-equation}
    \begin{split}
        \partial_t(u_\e \cdot \omega_\e) & + \divergence \bigg(u_\e (u_\e \cdot \omega_\e) + \bigg(p_\e - \frac{\abs{u_\e}^2}{2} \bigg)\omega_\e \bigg) = \omega_\e \cdot \divergence R_\e + u_\e \cdot \curl \divergence(R_\e) 
        \\ & = 2 \divergence(\omega_\e R_\e) - 2 \nabla \omega_\e \colon R_\e - \divergence (u \times \divergence R) 
        \\ & = 2 \divergence(\omega_\e R_\e) - 2 \nabla \omega_\e \colon R_\e + \divergence (\partial_l u \times R_\e^{\cdot, l}) - \divergence ( \partial_l (u_\e \times R_\e^{\cdot, l})) 
        \\ & = - 2 \nabla \omega_\e \colon R_\e + \divergence(2 \omega_\e R_\e + \partial_l u_\e \times R_\e^{\cdot, l} - \partial_l (u_\e \times R_\e^{\cdot, l}) ).  
    \end{split}
\end{equation}

\textsc{\underline{Step 4}}: We check that all the terms in the divergence at the right hand side go to $0$ in the sense of distributions. Fix a time slice $t \in [0,T]$. To begin, we prove that $R_\e(t) \to 0$ in $H^{\sfrac{1}{2}}_{\rm loc}(\Omega)$. By \cref{convergence of the squares} it holds that $u_\e(t) \otimes u_\e(t) \to u(t) \otimes u(t)$ in $H^{\sfrac{1}{2}}_{\rm loc}(\Omega)$ and, by \cref{l: approx by convolution} and \cref{H12 is algebra}, the same holds for $(u(t) \otimes u(t))_\e$. Thus, $R_\e(t) \to 0$ in $H^{\sfrac{1}{2}}_{\rm loc}(\Omega)$. By \cref{l: convergence in H^-1/2} we infer that $\partial_l u_\e^i(t) R_\e^{j,k}(t) \to 0$ in $\mathcal{D}'(\Omega)$, for any $i,j,k = 1,2,3$. Moreover, given a test function $\phi \in C^\infty_c(U \times I)$, take an open set $W$ such that $U \subset \joinrel \subset W \subset \joinrel \subset V$. Then, by \cref{l: approx by convolution}, \cref{l: convergence in H^-1/2}, \cref{H12 is algebra} and \cref{convergence of the squares} , for $\e \leq \e_0 = \min \{ \text{dist}(U,W^c), \text{dist}(W, V^c)\}$, it holds that 
\begin{align}
    \abs{\langle \partial_i u^l_\e(t) , R^{j,k}_\e(t) \phi(t) \rangle } & \lesssim \norm{ u_\e}_{H^{\sfrac{1}{2}  }(W)} \norm{R_\e(t) \phi(t)}_{H^{\sfrac{1}{2}  }(W)} 
    \\ & \lesssim \norm{u}_{H^{\sfrac{1}{2}  }(V)} \norm{u_\e(t)}_{H^{\sfrac{1}{2}  }(W)} \norm{u_\e(t)}_{L^\infty(W)} 
    \\ & \lesssim \norm{u(t)}_{H^{\sfrac{1}{2}  }(V)}^2 \norm{u(t)}_{L^\infty(V)} = : g(t), 
\end{align}
where the implicit constant depends also on $\phi, W$. Since $u \in L^2(I; H^{\sfrac{1}{2}}(V) )\cap L^\infty(V \times I)$, then we infer that $g \in L^1(I)$. Hence, we conclude that $ \partial_l u_\e^i R_\e^{j,k} \to 0$ in $\D'( \Omega \times (0,T))$. This argument proves that 
$$\divergence(2 \omega_\e R_\e + \partial_lu \times R_\e^{\cdot, l} ) \to 0 \text{ in } \D'(\Omega \times (0,T)). $$ 
Similarly, it can be checked that 
$$\divergence(\partial_l(u_\e \times R_\e^{\cdot, l})) \to 0 \text{ in } \D'(\Omega \times (0,T)), $$
thus the right hand side in \eqref{eq:main-equation} vanishes in the sense of distributions as $\e \to 0$. 

\textsc{\underline{Step 5}}: We consider the left hand side in \eqref{eq:main-equation}. By the same argument of the previous step, we deduce that 
\begin{equation}
    \partial_t (u_\e \cdot \omega_\e) + \divergence\bigg(u_\e (u_\e \cdot \omega_\e) - \frac{\abs{u_\e}^2}{2} \omega_\e\bigg)  \to \partial_t (u \cdot \omega) + \divergence\bigg(u (u \cdot \omega) - \frac{\abs{u}^2}{2} \omega\bigg) \qquad \text{in } \D'(\T^3 \times (0,T)). 
\end{equation}
To conclude, by \textsc{\underline{Step 1}}, it holds that $p \in L^1_{\rm loc}((0,T); H^{\sfrac{1}{2}}_{\rm loc}(\Omega)) $. Therefore, with the same argument of the previous step, we obtain that 
$$\divergence (p_\e \omega_\e) \to \divergence(p \omega) \qquad \text{in } \D'(\Omega \times (0,T)). $$
Finally, the term $2\nabla \omega_\e : R_\e$ (which a priori is not under control) defines a distribution in the limit by 
\begin{equation}
    \partial_t (u \cdot \omega) + \divergence\bigg(u(u \cdot \omega) + \bigg(p -\frac{\abs{u}^2}{2}\bigg) \omega \bigg) = - D[u] = - \lim_{\e \to 0} 2 \nabla \omega_\e : R_\e \qquad \text{in } \mathcal{D}'(\Omega \times (0,T)). 
\end{equation}
\end{proof}

\begin{proof} [Proof of \cref{cor: conservation}]
Fix open sets $I \subset \joinrel \subset (0,T), U \subset \joinrel \subset V \subset \joinrel \subset \Omega$. For $\e \leq \min\{ \text{dist}(U, V^c)\}$, by \cref{l: commutator estimates} we get that 
\begin{equation}
    \norm{\nabla \omega_\e(t) : R_\e(t)}_{L^1(U)} \lesssim [u(t)]_{B^{\sfrac{2}{3}}_{3,\infty}(U)}^3 \ell_{u(t), U} (\e)^3.  
\end{equation}
Since $\ell_{u(t), U}(\e) \to 0$ as $\e \to 0$ for a.e. $t \in (0,T)$, we conclude that $\nabla \omega_\e : R_\e \to 0$ in $L^1_{\rm loc}(\Omega \times (0,T))$ by dominated convergence. 
\end{proof}

\subsection{Total helicity balance} \label{sec:global}

In this section we study the total helicity. We recall that, given a distribution $F$ on an open set $\Omega \subset \R^d$, the support of $F$ is the complement of the largest open set $A$ such that $F(\phi) = 0$ for any $\phi \in C^\infty_c(A)$. Consider a weak solution $u \in L^2((0,T); H^{\sfrac{1}{2}}(\T^3)) \cap L^\infty(\T^3 \times (0,T))$ to \eqref{Euler} and let $D[u]$ be the helicity distribution defined by \cref{th:main-helicity}. Under these assumptions, since $\T^3$ has no physical boundary, then we can define the total helicity 
$$H(t) : = \langle \omega(t), u(t)\rangle.$$
Integrating in space \eqref{local helicity rough},
%and \eqref{local helicity rough} $D[u]$ can be integrated in space. 
 the time marginal of $D[u]$ is the distributional derivative of the total helicity $H$. Inspired by \cite{DRH22}*{Lemma 2.1} and \cite{LZ23}*{Theorem 1.1} we prove the following result. 

\begin{prop} \label{holder continuity vs supp dim}
Let $u \in  L^\infty ((0,T) ; B^\theta_{3, \infty})$ be a weak solution to \eqref{Euler} according to \cref{d: weak solution} with $\theta \in \left(\sfrac{1}{2} , \sfrac{2}{3}\right]$ such that $H' \neq 0$. Then 
$$ \text{dim}_{\mathcal{H}} (\text{Spt} (H')) \geq \frac{2 \theta -1 }{1 - \theta}.$$
\end{prop}

\begin{proof}
Denote by $\sigma := \frac{2 \theta -1 }{1 - \theta} $ and suppose by contradiction that $\text{dim}_{\mathcal{H}} (\text{Spt} (H')) < \sigma$. Then, by the definition of Hausdorff measure, for any $\e >0$ and any $\delta >0$ there exists $ N(\varepsilon , \delta) >0$ and a finitely many balls $\{ B_{r_i}(t_i) \}_{i=1}^N $ with $r_i \leq \delta$ and 
$$ \text{Spt} (H') \subset  \bigcup_{i =1 }^N B_{r_i}(t_i), \qquad \sum_{i =1}^N r_i^{\sigma} < \varepsilon. $$
Then, we prove that $H' \equiv 0$, which is a contradiction. By \cite{LZ23}*{Theorem 1.1}, $H$ agrees almost everywhere with a $C^\sigma$ function. Then, still denoting by $H$ the continuous representative, for any $t \in (0,T)$ we have that 
\begin{align}
    \abs{ H(t) - H(0)} \leq \sum_{i =1}^N \sup_{\tau , s \in B_{r_i}(t_i) } \abs{H(\tau   ) - H(s) } \leq  \norm{H}_{C^\sigma } \sum_{i=1}^N r_i^\sigma < \e \norm{H}_{C^\sigma }.
\end{align}
In the first inequality we used that $H$ is locally constant in the open set $  (\Spt (H'))^c $,  which is such that $(\cup_{i =1 }^N B_{r_i}(t_i) )^c \subset (\Spt (H'))^c$, and $(\cup_{i =1 }^N B_{r_i}(t_i) )^c$ is a closed set. Since $\e >0$ is arbitrary we conclude that $H \equiv H(0)$. 
\end{proof}

The remaining goal of this section is to prove \cref{thm:conservation-bdd-domain}, i.e. to study the variation of the total helicity in terms of the boundary flux of the vorticity. In particular, if the vorticity happens to be tangent to the boundary, then the total helicity is conserved. We recall that the condition of vanishing initial vorticity at the boundary is preserved by smooth Euler flows (see \cref{lem:computation-curl}).

We recall some results needed in the proof of \cref{thm:conservation-bdd-domain}.

\begin{prop} [\cite{Wahl92}*{Theorem 3.2}] \label{prop:wahl}
Let $\Omega \subset \R^3$ be a simply connected smooth bounded domain and $ p \in (1,  \infty)$. Assume $u \in W^{1,p} (\Omega)$ has the impermeability condition in the sense of trace operator. There exists an implicit constant depending only on $\Omega$ such that 
$$ \| \nabla u \|_{L^p(\Omega)} \lesssim \| \curl (u) \|_{L^p(\Omega)} + \| \divergence (u) \|_{L^p(\Omega)}.$$
\end{prop}

\begin{lem} [\cite{CS13}*{Lemma 3.5}] \label{lemma:Lpdensity} 
    Let $\Omega \subset \R^3$ be a smooth domain and $p \in (1, \infty)$. Denote by 
    $$X(\Omega) : = \{ u \in L^p (\Omega) : \curl(u ) \in L^p(\Omega), \, \divergence(u) \in L^p(\Omega), \, u \cdot n =0 \text{ on } \partial \Omega \},$$
    where the condition $u\cdot n = 0$ on $\partial \Omega$ is defined in the sense of the distributional normal traces. Then, $W^{1,p}(\Omega) \cap X(\Omega)$ is dense in $X(\Omega)$ with respect to the norm
    $$\| u \|_{X} : = \|u \|_{L^p(\Omega)} + \| \curl(u) \|_{L^p(\Omega)} + \| \divergence (u) \|_{L^p(\Omega)}. $$
\end{lem}

The regularity of the pressure in bounded domains is rather delicate. We mention \cites{DRLS23, DRLS24, BT22, BBT24} and the references therein for a detailed presentation on the topic. 

%\cref{thm:conservation-bdd-domain} is \cite[Theorem 1.1]{EAPS18}

%based the following \cite{Wahl92}*{Theorem 3.2} (see also \cite{J14}). 
\begin{comment}

\begin{prop} \label{p: est grad with curl and div}
Let $\Omega \subset \R^3$ be a $C^1$ simply connected open set and let $p \in (1,+\infty)$ and suppose that $\omega = u, \curl (u) \in L^\infty (\Omega)$ and 
$$ \int_{\partial \Omega} \omega \cdot \nu =0 \,,$$
where $\nu$ is the outer normal vector to $\Omega$,
then the following estimate holds for some constant $C>0$
\begin{equation}
   \| u \|_{W^{1,p} (\Omega)} \leq C \| \omega \|_{L^p (\Omega)} \,. 
\end{equation}
\end{prop}

\end{comment}
%\begin{prop} \label{p: est grad with curl and div}
%Let $\Omega \subset \R^3$ be a simply connected open set with $C^{2,1}$ boundary and let $p \in (1,+\infty)$. Then, for any $U \in W^{1,p}(\Omega)$ such that $U$ is tangent to $\partial \Omega$, it holds that 
%\begin{equation}
  % \| \nabla U \|_{L^p (\Omega)} \leq C \| \curl (U) \|_{L^p (\Omega)}   +\| \divergence (U) \|_{L^p (\Omega)} \qquad \forall U \in W^{1,p}(\Omega), 
%\end{equation}
%\end{prop}
\begin{lem} \label{lemma:omega-smooth}
Let $\Omega \subset \R^3$ be a simply connected bounded open set with smooth boundary. Let $(u,p) \in L^\infty(\Omega \times (0,T)) $ be a weak solution to \eqref{Euler} according to \cref{d: weak solution} such that $\omega = \curl (u) \in L^\infty(\Omega \times (0,T))$. Assume that $p(t)$ has zero average for a.e. $t$. Then, it holds that $u,p \in L^\infty((0,T); C^\alpha(\Omega))$ for any $\alpha\in (0,1)$ and in particular  $u \in L^\infty((0,T); H^{\sfrac{1}{2}}(\Omega))$. 
\end{lem} 

\begin{proof}
We fix a time $t \in (0,T)$ and we prove all the estimates with respect to the spatial variables uniformly in time. Fix $p \in (1, +\infty)$. By Lemma \ref{lemma:Lpdensity} there exists $u_\varepsilon \to u$ as $\varepsilon \to 0$ with respect to the metric $X$ introduced in \cref{lemma:Lpdensity}. Applying \cref{prop:wahl} we find an implicit constant depending only on $\Omega$ such that  
$$ \| \nabla u_\varepsilon \|_{L^p(\Omega) } \lesssim \| \curl (u_\varepsilon) \|_{L^p(\Omega)} +  \| \divergence (u_\varepsilon)\|_{L^p(\Omega)} $$
for any $\varepsilon>0$. The right hand side converges to $\| \omega \|_{L^p(\Omega)} $ and on the left hand side by the lower semicontinuity of the norm we deduce that
$$ \| \nabla u  \|_{L^p(\Omega) } \lesssim \| \omega  \|_{L^p(\Omega)}.$$
Since $p \in (1, \infty)$ is arbitrary we deduce that $u \in L^\infty((0,T); C^\alpha(\Omega))$ for any $\alpha \in (0,1)$. Then, it is immediate to deduce also that $u \in L^2((0,T); H^{\sfrac{1}{2}}(\Omega)) $. For the pressure, for any $\alpha \in \left( \sfrac{1}{2},1\right)$ by the discussion in \cite{DRLS24}*{Section 2.1} and \cite{DRLS23}*{Theorem 1.1} it follows that $p (\cdot, t) \in C^{1,2\alpha-1}(\Omega)$ and it holds 
\begin{equation}
    \norm{p(\cdot, t)}_{C^{1,2\alpha-1}(\Omega)} \lesssim \norm{u(\cdot, t)}_{C^\alpha(\Omega)}^2 \lesssim \norm{u(\cdot, t)}_{L^\infty(\Omega)}^2 + \norm{\omega(\cdot, t)}_{L^\infty(\Omega)}^2. 
\end{equation}
\end{proof}

\begin{proof}[Proof of \cref{thm:conservation-bdd-domain}]
Given $r>0$, we define $\chi_r \in W^{1, \infty} (\Omega)$ as in \cref{prop:boundary}. Pick any test function $\alpha \in C^\infty_c((0,T))$. Since $\chi_r$ is Lipschitz continuous and vanishing on $\partial \Omega$, it holds that $D[u]$ can be tested against $\phi_r(x,t) \coloneqq \chi_r(x) \alpha(t)$. Under our assumptions, by \cref{lemma:omega-smooth} and \cref{cor: conservation}, $D[u]$ is well defined and it vanishes. Then, we have that
\begin{align}
    \underbrace{\int_0^T \alpha'  \int_{\Omega} \omega \cdot u \chi_r \, dx  \, dt}_{I_{r}} + \underbrace{\int_0^T \alpha  \int_{\Omega}   (u \cdot \omega) u \cdot \nabla \chi_r \, dx  \, dt}_{II_r} + \underbrace{\int_0^T \alpha  \int_{\Omega}  \left( p - \frac{\abs{u}^2}{2} \right) \omega \cdot \nabla \chi_r \, dx  \, dt}_{III_r}  =  0. 
\end{align}
Since $\chi_r \to 0$ pointwise, $ 0 \leq \chi_r \leq 1$ and $u, \omega \in L^\infty_{x,t}$, by dominated convergence we  have
\begin{equation}
    \lim_{r \to 0} I_r = \int_0^T \alpha' (t) H(t) \, dt. 
\end{equation}
We recall that $u(t)$ is tangent to the boundary in the sense of \eqref{eq: impermeability condition} and $u(t) \in C^0(\overline{\Omega})$ for a.e. $t$ by \cref{lemma:omega-smooth}. Thus, $u(t)$ has a normal Lebesgue trace on $\partial \Omega$ according to \cref{def:normal} and by the discussion in \cref{ss: boundary trace} and \cite{CDRIN24}*{Theorem 1.4} it holds that $u(t)\cdot n \equiv 0$ on $\partial \Omega$ for a.e. $t \in (0,T)$. Since $u, \omega \in L^\infty_{x,t}$, by \cref{prop:boundary} and dominated convergence we infer that $II_r \to 0$ as $r \to 0$. Similarly, since $u(t), p(t)$ have a full trace on $\partial \Omega$ by \cref{lemma:omega-smooth} and $\omega(t)$ is assumed to have normal Lebesgue trace $\omega_n(t)$ according to \cref{def:normal}, by \cref{prop:boundary} and dominated convergence we infer that 
\begin{equation}
    \lim_{r \to 0} III_r = - \int_0^T \alpha(t) \left[  \int_{\partial \Omega} \left( \frac{\abs{u}^2}{2} - p \right) \omega_n \, d \mathcal{H}^2(x) \right] \, dt. 
\end{equation}
\end{proof}

\subsection*{Acknowledgements} MI is partially funded by the SNF grant FLUTURA: Fluids, Turbulence, Advection No.
212573. MS is supported by the Swiss State Secretariat for Education, Research and lnnovation (SERI) under contract number MB22.00034 through the project TENSE. The authors thank the anonymous referee for careful reading the paper and several useful comments.

\bibliographystyle{plain}
\bibliography{biblio}

\begin{bibdiv}
\begin{biblist}

\bib{ACM05}{article}{
      author={Ambrosio, Luigi},
      author={Crippa, Gianluca},
      author={Maniglia, Stefania},
       title={Traces and fine properties of a {$BD$} class of vector fields and
  applications},
        date={2005},
        ISSN={0240-2963,2258-7519},
     journal={Ann. Fac. Sci. Toulouse Math. (6)},
      volume={14},
      number={4},
       pages={527\ndash 561},
         url={http://afst.cedram.org/item?id=AFST_2005_6_14_4_527_0},
      review={\MR{2188582}},
}

\bib{AFP00}{book}{
      author={Ambrosio, Luigi},
      author={Fusco, Nicola},
      author={Pallara, Diego},
       title={Functions of bounded variation and free discontinuity problems},
      series={Oxford Mathematical Monographs},
   publisher={The Clarendon Press, Oxford University Press, New York},
        date={2000},
        ISBN={0-19-850245-1},
      review={\MR{1857292}},
}

\bib{CS13}{article}{
      author={Amrouche, Ch\'erif},
      author={Seloula, Nour El~Houda},
       title={{$L^p$}-theory for vector potentials and {S}obolev's inequalities
  for vector fields: application to the {S}tokes equations with pressure
  boundary conditions},
        date={2013},
        ISSN={0218-2025,1793-6314},
     journal={Math. Models Methods Appl. Sci.},
      volume={23},
      number={1},
       pages={37\ndash 92},
         url={https://doi.org/10.1142/S0218202512500455},
      review={\MR{2997467}},
}

\bib{BBT24}{article}{
      author={Bardos, C.},
      author={Boutros, D.~W.},
      author={Titi, E.~S.},
       title={H\"older regularity of the pressure for weak solutions of the
  3{D} {E}uler equations in bounded domains},
        date={2024},
     journal={Prepring available at
  \href{https://arxiv.org/abs/2304.01952}{arxiv:2304.01952}},
}

\bib{BT22}{article}{
      author={Bardos, C.~W.},
      author={Titi, E.~S.},
       title={{$C^{0,\alpha}$} boundary regularity for the pressure in weak
  solutions of the 2d {E}uler equations},
        date={2022},
        ISSN={1364-503X,1471-2962},
     journal={Philos. Trans. Roy. Soc. A},
      volume={380},
      number={2218},
       pages={Paper No. 20210073, 15},
         url={https://doi.org/10.1098/rsta.2021.0073},
      review={\MR{4395520}},
}

\bib{BT24}{article}{
      author={Boutros, Daniel~W.},
      author={Titi, Edriss~S.},
       title={On the conservation of helicity by weak solutions of the {3D}
  {E}uler and inviscid {MHD} equations},
        date={2024},
     journal={Prepring available at
  \href{https://arxiv.org/abs/2410.00813}{arxiv:2410.00813}},
}

\bib{BV19}{article}{
      author={Buckmaster, Tristan},
      author={Vicol, Vlad},
       title={Nonuniqueness of weak solutions to the {N}avier-{S}tokes
  equation},
        date={2019},
        ISSN={0003-486X,1939-8980},
     journal={Ann. of Math. (2)},
      volume={189},
      number={1},
       pages={101\ndash 144},
         url={https://doi.org/10.4007/annals.2019.189.1.3},
      review={\MR{3898708}},
}

\bib{CCFS08}{article}{
      author={Cheskidov, A.},
      author={Constantin, P.},
      author={Friedlander, S.},
      author={Shvydkoy, R.},
       title={Energy conservation and {O}nsager's conjecture for the {E}uler
  equations},
        date={2008},
        ISSN={0951-7715},
     journal={Nonlinearity},
      volume={21},
      number={6},
       pages={1233\ndash 1252},
         url={https://doi.org/10.1088/0951-7715/21/6/005},
      review={\MR{2422377}},
}

\bib{CDRS22}{article}{
      author={Colombo, Maria},
      author={De~Rosa, Luigi},
      author={Sorella, Massimo},
       title={Typicality results for weak solutions of the incompressible
  {N}avier-{S}tokes equations},
        date={2022},
        ISSN={1292-8119,1262-3377},
     journal={ESAIM Control Optim. Calc. Var.},
      volume={28},
       pages={Paper No. 38, 24},
         url={https://doi.org/10.1051/cocv/2022038},
      review={\MR{4438714}},
}

\bib{CET94}{article}{
      author={Constantin, Peter},
      author={E, Weinan},
      author={Titi, Edriss~S.},
       title={Onsager's conjecture on the energy conservation for solutions of
  {E}uler's equation},
        date={1994},
        ISSN={0010-3616},
     journal={Comm. Math. Phys.},
      volume={165},
      number={1},
       pages={207\ndash 209},
         url={http://projecteuclid.org/euclid.cmp/1104271041},
      review={\MR{1298949}},
}

\bib{CDRIN24}{article}{
      author={Crippa, G.},
      author={Rosa, L.~De},
      author={Inversi, M.},
      author={Nesi, M.},
       title={Normal traces and applications to continuity equations on bounded
  domains},
        date={2024},
     journal={Preprint available at
  \href{https://arxiv.org/pdf/2405.11486}{arXiv:2405.11486}},
}

\bib{DLL09}{article}{
      author={De~Lellis, Camillo},
      author={Sz\'{e}kelyhidi, L\'{a}szl\'{o}, Jr.},
       title={The {E}uler equations as a differential inclusion},
        date={2009},
        ISSN={0003-486X},
     journal={Ann. of Math. (2)},
      volume={170},
      number={3},
       pages={1417\ndash 1436},
         url={https://doi.org/10.4007/annals.2009.170.1417},
      review={\MR{2600877}},
}

\bib{DLL13}{article}{
      author={De~Lellis, Camillo},
      author={Sz\'{e}kelyhidi, L\'{a}szl\'{o}, Jr.},
       title={Dissipative continuous {E}uler flows},
        date={2013},
        ISSN={0020-9910},
     journal={Invent. Math.},
      volume={193},
      number={2},
       pages={377\ndash 407},
         url={https://doi.org/10.1007/s00222-012-0429-9},
      review={\MR{3090182}},
}

\bib{DR19}{article}{
      author={De~Rosa, Luigi},
       title={On the helicity conservation for the incompressible {E}uler
  equations},
        date={2020},
        ISSN={0002-9939},
     journal={Proc. Amer. Math. Soc.},
      volume={148},
      number={7},
       pages={2969\ndash 2979},
         url={https://doi.org/10.1090/proc/14952},
      review={\MR{4099784}},
}

\bib{DRH22}{article}{
      author={De~Rosa, Luigi},
      author={Haffter, Silja},
       title={Dimension of the singular set of wild {H}\"{o}lder solutions of
  the incompressible {E}uler equations},
        date={2022},
        ISSN={0951-7715},
     journal={Nonlinearity},
      volume={35},
      number={10},
       pages={5150\ndash 5192},
      review={\MR{4500861}},
}

\bib{DRINV23}{article}{
      author={De~Rosa, Luigi},
      author={Inversi, Marco},
       title={Dissipation in {O}nsager's critical classes and energy
  conservation in {$BV\cap L^\infty$} with and without boundary},
        date={2024},
        ISSN={0010-3616,1432-0916},
     journal={Comm. Math. Phys.},
      volume={405},
      number={1},
       pages={Paper No. 6, 34},
         url={https://doi.org/10.1007/s00220-023-04922-3},
      review={\MR{4691856}},
}

\bib{DRLS23}{article}{
      author={De~Rosa, Luigi},
      author={Latocca, Micka\"el},
      author={Stefani, Giorgio},
       title={On double {H}\"older regularity of the hydrodynamic pressure in
  bounded domains},
        date={2023},
        ISSN={0944-2669,1432-0835},
     journal={Calc. Var. Partial Differential Equations},
      volume={62},
      number={3},
       pages={Paper No. 85, 31},
         url={https://doi.org/10.1007/s00526-023-02432-7},
      review={\MR{4537425}},
}

\bib{DRLS24}{article}{
      author={De~Rosa, Luigi},
      author={Latocca, Micka\"el},
      author={Stefani, Giorgio},
       title={Full double {H}\"older regularity of the pressure in bounded
  domains},
        date={2024},
        ISSN={1073-7928,1687-0247},
     journal={Int. Math. Res. Not. IMRN},
      number={3},
       pages={2511\ndash 2560},
         url={https://doi.org/10.1093/imrn/rnad197},
      review={\MR{4702283}},
}

\bib{DNPV12}{article}{
      author={Di~Nezza, Eleonora},
      author={Palatucci, Giampiero},
      author={Valdinoci, Enrico},
       title={Hitchhiker's guide to the fractional {S}obolev spaces},
        date={2012},
        ISSN={0007-4497},
     journal={Bull. Sci. Math.},
      volume={136},
      number={5},
       pages={521\ndash 573},
         url={https://doi.org/10.1016/j.bulsci.2011.12.004},
      review={\MR{2944369}},
}

\bib{DPL89}{article}{
      author={DiPerna, R.~J.},
      author={Lions, P.-L.},
       title={Ordinary differential equations, transport theory and {S}obolev
  spaces},
        date={1989},
        ISSN={0020-9910,1432-1297},
     journal={Invent. Math.},
      volume={98},
      number={3},
       pages={511\ndash 547},
         url={https://doi.org/10.1007/BF01393835},
      review={\MR{1022305}},
}

\bib{DR00}{article}{
      author={Duchon, Jean},
      author={Robert, Raoul},
       title={Inertial energy dissipation for weak solutions of incompressible
  {E}uler and {N}avier-{S}tokes equations},
        date={2000},
        ISSN={0951-7715},
     journal={Nonlinearity},
      volume={13},
      number={1},
       pages={249\ndash 255},
         url={https://doi.org/10.1088/0951-7715/13/1/312},
      review={\MR{1734632}},
}

\bib{Ey94}{article}{
      author={Eyink, G.~L.},
       title={Energy dissipation without viscosity in ideal hydrodynamics. {I}.
  {F}ourier analysis and local energy transfer},
        date={1994},
        ISSN={0167-2789,1872-8022},
     journal={Phys. D},
      volume={78},
      number={3-4},
       pages={222\ndash 240},
         url={https://doi.org/10.1016/0167-2789(94)90117-1},
      review={\MR{1302409}},
}

\bib{Fr95}{book}{
      author={Frisch, U.},
       title={Turbulence},
   publisher={Cambridge University Press, Cambridge},
        date={1995},
        ISBN={0-521-45103-5},
        note={The legacy of A. N. Kolmogorov},
      review={\MR{1428905}},
}

\bib{Is18}{article}{
      author={Isett, P.},
       title={A proof of {O}nsager's conjecture},
        date={2018},
        ISSN={0003-486X,1939-8980},
     journal={Ann. of Math. (2)},
      volume={188},
      number={3},
       pages={871\ndash 963},
         url={https://doi.org/10.4007/annals.2018.188.3.4},
      review={\MR{3866888}},
}

\bib{GL23}{book}{
      author={Leoni, Giovanni},
       title={A first course in fractional {S}obolev spaces},
      series={Graduate Studies in Mathematics},
   publisher={American Mathematical Society, Providence, RI},
        date={[2023] \copyright2023},
      volume={229},
        ISBN={[9781470468989]; [9781470472535]; [9781470472528]},
         url={https://doi.org/10.1090/gsm/229},
      review={\MR{4567945}},
}

\bib{LZ23}{article}{
      author={Liu, J.},
      author={Zhao, Y.},
       title={H\"older regularity of helicity for the incompressible flows},
        date={2023},
        ISSN={1422-6928,1422-6952},
     journal={J. Math. Fluid Mech.},
      volume={25},
      number={1},
       pages={Paper No. 16, 10},
         url={https://doi.org/10.1007/s00021-022-00760-w},
      review={\MR{4534511}},
}

\bib{NV23}{article}{
      author={Novack, M.},
      author={Vicol, V.},
       title={An intermittent {O}nsager theorem},
        date={2023},
        ISSN={0020-9910,1432-1297},
     journal={Invent. Math.},
      volume={233},
      number={1},
       pages={223\ndash 323},
         url={https://doi.org/10.1007/s00222-023-01185-6},
      review={\MR{4601999}},
}

\bib{T83}{book}{
      author={Triebel, Hans},
       title={Theory of function spaces},
      series={Monographs in Mathematics},
   publisher={Birkh\"auser Verlag, Basel},
        date={1983},
      volume={78},
        ISBN={3-7643-1381-1},
         url={https://doi.org/10.1007/978-3-0346-0416-1},
      review={\MR{781540}},
}

\bib{GKN23}{article}{
      author={V.~Giri, H.~Kwon},
      author={Novack, M.},
       title={The {$L^3$}-based strong {O}nsager theorem},
        date={2023},
     journal={Prepring available at
  \href{https://arxiv.org/abs/2305.18509}{arxiv:2305.18509}},
}

\bib{Wahl92}{article}{
      author={von Wahl, Wolf},
       title={Estimating {$\nabla u$} by {${\rm div}\, u$} and {${\rm curl}\,
  u$}},
        date={1992},
        ISSN={0170-4214,1099-1476},
     journal={Math. Methods Appl. Sci.},
      volume={15},
      number={2},
       pages={123\ndash 143},
         url={https://doi.org/10.1002/mma.1670150206},
      review={\MR{1149300}},
}

\end{biblist}
\end{bibdiv}

\end{document}